\documentclass[12pt]{amsart}
\usepackage{amstext,amsfonts,amssymb,amscd,amsbsy,amsmath,verbatim}
\usepackage[alphabetic, lite]{amsrefs} 
\usepackage{ifthen, fullpage}
\usepackage{color,tikz}
\usepackage{amsthm}
\usepackage{latexsym}
\usepackage[all]{xy}
\usepackage{enumerate}

\newtheorem{lemma}{Lemma}[section]
\newtheorem{theorem}[lemma]{Theorem}

\newtheorem{prop}[lemma]{Proposition}
\newtheorem{question}[lemma]{Question}
\newtheorem{cor}[lemma]{Corollary}

\newtheorem{claim*}{Claim}
\newtheorem{thm}[lemma]{Theorem}
\newtheorem{defn}[lemma]{Definition}
\newtheorem*{lemmaA2}{Lemma A.2}

\theoremstyle{remark}
\newtheorem{remark}[lemma]{Remark}
\newtheorem{example}[lemma]{Example}

\newcommand{\im}{\operatorname{im}}

\newcommand{\ZZ}{{\mathbb Z}}

\newcommand{\Tor}{\operatorname{Tor}}

\newcommand{\Hom}{\operatorname{Hom}} 
\newcommand{\Ext}{\operatorname{Ext}} 

\newcommand{\coker}{\operatorname{coker}}

\newcommand{\NF}{F(\mathbf{f})}

\newcommand{\defi}[1]{\textsf{#1}} 

\newcommand{\BQ}{\ensuremath{B_\mathbb Q}}
\newcommand{\Bmod}{\ensuremath{B_{\operatorname{mod}}}}
\newcommand{\Bint}{\ensuremath{B_\text{int}}}

\def\BS{Boij-S\"oderberg }

\def\triplearrow#1#2{{\entrymodifiers={+!!<0pt,\fontdimen22\textfont2>}
 \def\objectstyle{\displaystyle}
 \xymatrix{   \ar@<0.7ex>^-{#1}[r]\ar[r]\ar@<-0.7ex>_-{#2}[r] & }}}

\title{Filtering Free Resolutions}
\author{David Eisenbud}
\address{Department of Mathematics, University of California,
	Berkeley, CA 94720, USA}
\email{eisenbud@math.berkeley.edu}
\author{Daniel Erman}
\address{Department of Mathematics, University of Wisconsin-Madison,
	Madison, WI 53706, USA}
\email{derman@math.wisc.edu}
\urladdr{http://www.math.wisc.edu/~derman/}
\author{Frank-Olaf Schreyer}
\address{Mathematik und Informatik, Universit\"{a}t des Saarlandes, Campus E2 4, D-66123 Saarbr\"{u}cken, Germany}
\email{schreyer@math.uni-sb.de}

\thanks{The first author was partially supported by an NSF grant, and the second author was partially supported by an NDSEG fellowship and an NSF postdoctoral fellowship.}


\begin{document}
\maketitle
\begin{abstract}
A recent result of Eisenbud-Schreyer and Boij-S\"oderberg proves that the Betti diagram of any graded module decomposes as a positive rational linear combination of pure diagrams.  When does this numerical decomposition correspond to an actual filtration of the minimal free resolution?  Our main result gives a sufficient condition for this to happen.  We apply it to show the non-existence of free resolutions with some plausible-looking Betti diagrams and to study the semigroup of quiver representations of the simplest ``wild'' quiver.
\end{abstract}

\section{Introduction}
\noindent  Let $k$ be a  field, let $S:=k[x_1, \dots, x_n]$ be the polynomial ring, and let $M$ be a finitely generated graded $S$-module. We write
\[
\xymatrix{
F^{M}_{\phantom{p}}:& 0\ar[r]&F^{M}_p \ar[r]^{\phi_p}& \dots\ar[r]^{\phi_2}& F^{M}_1 \ar[r]^{\phi_1} &F^{M}_0
}
\]
for the graded minimal free resolution of $M$.  We define $\beta_{i,j}(F^{M}) = \beta_{i,j}(M)$ by the formula
\[
F^M_i=\bigoplus_{j\in \mathbb Z} S(-j)^{\beta_{i,j}(M)}.
\]
The underlying question of this paper is as follows.
\begin{question}\label{question:beyond0}
When does a knowledge of the numbers $\beta_{i,j}$
imply that the module $M$ decomposes as a direct sum? More generally,
when can we deduce from the Betti numbers that the $M$ has a submodule
$M'$ whose free resolution $F^{M'}$ is a summand, term by term, of $F^{M}$?
\end{question}

We will say that a submodule $M' \subset M$ is \defi{cleanly embedded} if it satisfies the condition in the last sentence of the question---that is, if the natural map 
$$
{\rm Tor}_{i}^{S}(M', k) \to{\rm Tor}_{i}^{S}(M, k)
$$
is a monomorphism for every $i$. Of course any summand is cleanly embedded.

Here is a well-known example where knowledge of the $\beta_{i,j}$ allows us to predict a summand. Suppose that $M$ is zero in negative degrees, that is, $\beta_{0,j}(M) = 0$ for $j<0$. If $\beta_{n,n}(M)= b$ then $M$ contains $\bigl(S/(x_{1},\dots,x_{n})\bigr)^{b}$ as a direct summand. (Reason: $\beta_{n,n}(M)$ is, by local duality, equal to the component of the socle of $M$ in degree $0$.) 

Question~\ref{question:beyond0} has a special interest in light of \BS Theory:
The conjecture of Boij and S\"oderberg, proven by Eisenbud and Schreyer in \cite{eis-schrey1} and then extended in \cite{boij-sod2}, says that the Betti diagram of $M$ can be written
uniquely as a positive rational linear combination 
$$
\beta(M) = \sum_{t = 0}^{s}c_{t}\pi_{d^{t}}
$$
of \emph{pure} Betti diagrams $\pi_{d^{t}}$ where the \emph{degree sequences} $d^{t}$ 
satisfy $d^0 < d^1 < \ldots < d^s$. Here a degree sequence is an element
$$
d= (d_{0}, \dots, d_{n})\in (\ZZ\cup \{\infty\})^{n+1} \text{ with } d_i+1\le d_{i+1} \text{ for all } i ,
$$
and the (rational) Betti diagram $\pi_{d}$ is given by
\begin{equation}\label{eqn:pid}
\beta_{i,j}(\pi_{d}) = \begin{cases} 0&j\ne d_i\\
\prod_{k\ne i, d_k<\infty}\frac{1}{|d_i-d_{k}|} &j=d_i,
\end{cases}
\end{equation}
and $d^{t}\le d^{t+1}$ means that $d^{t}_{i}\leq  d^{t+1}_{i}$ for
every $i$.
(see \S\ref{sec:background} for the definition of a pure diagram and a summary of the necessary part of \BS theory).

With this result in mind it is natural to refine Question~\ref{question:beyond0} and ask:

\begin{question}\label{question:beyond}\label{question:beyond1}
When does the decomposition of the Betti diagram of a graded module $M$ into pure diagrams arise from some filtration of $M$ by cleanly embedded submodules?

In particular, when does the Betti diagram $c_{0}\pi_{d^{0}}$ correspond to the
resolution of a cleanly embedded submodule $M'\subset M$?
\end{question}

Certainly such a submodule $M'$  \emph{does not} always exist:  often the numbers
$\beta_{i,j}(c_{0}\pi_{d^{0}})$ are not even integers, and there are subtler reasons
as well (see Example~\ref{ex:first filtration} and \S\ref{sec:beyond}). 
However, our main result says such a module $M'$ \emph{does} exist when 
$d^{0}$ is ``sufficiently separate'' from the rest of the $d^{t}$. To make this precise, we write
\begin{equation}\label{defn:ll}\tag{*}
d^0\ll d^1 \text{ if } d^0<d^1, d^0_2\leq d^1_1 \text{ and if } d^0_i \leq d^1_2+i-1\text{ for } i>2.
\end{equation}

\begin{thm}[(Existence of a cleanly embedded pure submodule)]\label{thm:puresexist}
Let $\dim(S)\geq 2$ and let $M$ be a finite length graded $S$-module with Boij-S\"oderberg decomposition 
$$
\beta(M)=\sum_{i=0}^s c_{i}\pi_{d^i}.
$$
\begin{enumerate}
\item\label{thm:puresexist:1} If  $d^0 \ll d^{1}$, then there is a cleanly embedded 
submodule $M'\subset M$ with $\beta(M') = c_{0}\pi_{d^0}$.
In particular, the diagram $c_{0}\pi_{d^0}$ has integer entries. 
\item\label{thm:puresexist:2} If $d^0 \ll d^{1}$ and $d^{0}_{n}-n < d^{1}_{1}$, then  $M'$ is a direct summand of $M$.
\end{enumerate}
\end{thm}

\noindent With corresponding hypotheses on all $d^{i}$, we obtain a full clean filtration (as in Definition~\ref{defn:pure filtration}).

\begin{cor}\label{cor:puresexist-iterated}
If, with hypotheses as in Theorem~\ref{thm:puresexist}, 
$d^{0}\ll d^{1}\ll \cdots \ll d^{s}$, then $M$ admits a filtration 
$0= M^{0} \subset \cdots \subset M^{s}\subset M^{s+1}$
by cleanly embedded submodules $M^{i}$
such that
$
\beta(M^{i+1}/M^{i})=c_i\pi_d^{i}.
$
\end{cor}

In the following, and in the rest of the paper, we write the \defi{Betti diagram} of $M$, $\beta(M)$, as a matrix whose entry in column $i$ and row $i+j$ is $\beta_{i,j}(M)$.  In examples, we follow the convention that the upper left entry of $\beta(M)$ corresponds to $\beta_{0,0}(M)$.

\begin{example}
Let $S=k[x,y,z]$.  If $M$ is any module with
\[
\beta(M)=\begin{pmatrix}
4&8&6&-\\
-&6&8&4
\end{pmatrix}=\begin{pmatrix}
3&8&6&-\\
-&-&-&1
\end{pmatrix}
+
\begin{pmatrix}
1&-&-&-\\
-&6&8&3
\end{pmatrix},
\]
then, since the corresponding degree sequences are $d^0=(0,1,2,4)$ and $d^1=(0,2,3,4)$, Theorem~\ref{thm:puresexist}\eqref{thm:puresexist:2} implies that $M$ splits as $M=M^1\oplus M^2$ with
\[
\beta(M^1)=\begin{pmatrix}
3&8&6&-\\
-&-&-&1
\end{pmatrix}
\quad \text{ and } 
\beta(M^2)=\begin{pmatrix}
1&-&-&-\\
-&6&8&3
\end{pmatrix}.
\]
\end{example}

The technique we develop to prove Theorem~\ref{thm:puresexist} actually yields the result in more general (but harder to formulate) circumstances; see \S\ref{sec:beyond}.

\subsection*{Application: The Insufficiency of Integrality}
One application of  Theorem~\ref{thm:puresexist} is to prove the non-existence of 
resolutions having otherwise plausible-looking Betti diagrams:
\begin{prop}\label{prop:sparserays}
Let $p\in \ZZ$ be any prime.  Then there exists a diagram $D$ with integral entries, such that $cD$ is the Betti diagram of a module if and only if $c$ is divisible by $p$. 
\end{prop}
\noindent This result simultaneously strengthens parts $(2), (3)$ and $(4)$ of~\cite[Thm.\ 1.6]{erman-semigroup}. Its proof is given in \S\ref{sec:pathologies}. The following question, posed 
in \cite[Conjecture 6.1]{efw}, remains open:  do all but finitely many integral points on a ray of \emph{pure} diagrams correspond to the Betti diagram of a module?

\begin{example}\label{ex:first filtration}
There is no graded module $M$ of finite length with Betti diagram
\[
D:=
\begin{pmatrix}
2 & 3 & 2 & - \\
- & 3 &3 &-\\
- & 2&3&2
\end{pmatrix}.
\]
Reason: The \BS decomposition of $D$ is
\[
D = 
\frac{1}{5}
\begin{pmatrix}
6&15&10&-\\
-&-&-&-\\
-&-&-&1
\end{pmatrix}
+
\frac{3}{5}
\begin{pmatrix}
1&-&-&-\\
-&5&5&-\\
-&-&-&1
\end{pmatrix}
+
\frac{1}{5}
\begin{pmatrix}
1&-&-&-\\
-&-&-&-\\
-&10&15&6
\end{pmatrix}
\]
The corresponding degree sequences are $d^0=(0,1,2,5), d^1=(0,2,3,5)$ and $d^2=(0,3,4,5)$, so  Theorem~\ref{thm:puresexist} implies that a module with
Betti diagram $D$ would admit a cleanly embedded submodule $M'$ with
betti diagram
\[
\beta(M')=\frac{1}{5}
\begin{pmatrix}
6&15&10&-\\
-&-&-&-\\
-&-&-&1
\end{pmatrix} = 
\begin{pmatrix}
\frac{6}{5}&3&2&-\\
-&-&-&-\\
-&-&-&\frac{1}{5}
\end{pmatrix}.
\]
This is absurd, since the entries of the diagram are not integers.

Now consider the diagrams $cD$, where $c$ is a rational number. The same
argument implies that these are not Betti diagrams of modules of finite length unless
$c$ is an integral multiple of 5. On the other hand, 
if $R:= k[x,y,z]/(x,y,z)^{3}, \omega_R(3)$ is the twisted dual of $R$, and $R':=k[x,y,z]/(x^{2},y^{2},z^{2}-xy, xz, yz)$, then
\begin{equation}\label{eqn:5D}
\beta(R\oplus \omega_R(3) \oplus R'^{\oplus 3} )=\begin{pmatrix}
10 & 15 & 10 & - \\
- & 15 &15 &-\\
- & 10&15&10
\end{pmatrix}
=
5D.
\end{equation}
We conclude that $cD$ is the Betti diagram of a module of finite length if and only if $c$ is an integral
multiple of $5$.\qed
\end{example}

\smallskip\noindent{\bf Application: invariants of the representations of $\bullet\triplearrow{}{}\bullet$}
\label{ex:2x3}
It was proven in \cite[Thm\ ~1.3]{erman-semigroup} that the semigroup of all Betti diagrams of modules
with bounded regularity and generator degrees is finitely generated, and the generators
were worked out in some small examples. In those cases the semigroup coincides with
the set of integral points in the positive rational cone generated by
the Betti diagrams of modules. With the added power of 
Theorem~\ref{thm:puresexist} we can determine the generators in the first case where this
does not happen: the case of modules over $k[x,y,z]$ having only two nonzero graded components,
$M = M_{0}\oplus M_{1}$. 

This case has an  interpretation in the representation theory of quivers. Consider representations over $k$ of 
the quiver with three arrows:
$$
Q: \bullet \triplearrow{}{} \bullet
$$
The problem of classifying representations of $Q$ up to isomorphism (or, equivalently, classifying triples of matrices up to simultaneous equivalence) is famously of ``wild type''; the variety of classes of representations with a given dimension vector $D:=(\dim M_{0}, \dim M_{1}) $ has dimension that grows with $D$, and many components. 

The Betti diagram of $M$ provides a discrete invariant of such a representation.
The (Castelnuovo-Mumford) regularity of $M$ is 1, so the Betti diagram has the form
$$
\beta(M)=\begin{pmatrix}
\beta_{0,0}& \beta_{1,1}& \beta_{2,2}& \beta_{3,3} \\
\beta_{0,1}& \beta_{1,2}& \beta_{2,3}& \beta_{3,4} \\
\end{pmatrix}.
$$
Some of the numbers in this diagram are easy to understand: for example, $\beta_{3,3}$ is the dimension of the common kernel of the three matrices, and $\beta_{0,1}$ is the dimension of $M_{1}$ modulo the sum of the images of the matrices. Passing to an obvious subquotient, therefore, we may assume that
$\beta_{3,3} = \beta_{0,1} = 0$. In this case $\beta_{0,0} = \dim M_{0}$ and $\beta_{1,1} = \dim M_{1}-3\beta_{0,0}$ are determined by the dimension vector $D$, as are $\beta_{3,4}$ and $\beta_{2,3}$ and the difference $\beta_{1,2}-\beta_{2,2}$. 

However, the value of $\beta_{2,2}$ is a more subtle invariant, semicontinuous on the family of equivalence classes of representations. In \S\ref{sec:gensBmod} we determine the semigroup of Betti diagrams $\beta(M)$ that come from representations of $Q$.

\smallskip\noindent{\bf A monotonicity principle and the proof of Theorem~\ref{thm:puresexist}.}  In order to prove Theorem~\ref{thm:puresexist}, we must construct an appropriate submodule of $M$ based only on the information contained in the Betti diagram of $M$.  Our construction is based on the notion of a \defi{numerical subcomplex}.

\begin{defn}\label{defn:numerical subcomplex}
A \defi{numerical subcomplex} of a minimal free resolution $F^{M}$ is a subcomplex $G$ ``whose existence is evident from the Betti diagram $\beta(F^{M})$'' in the sense that there is a sequence of integers $\alpha_{i}$ such that each
 $G_{i}$ consists of all the summands of $F^{M}_{i}$ generated in degrees $< \alpha_{i}$, and each $F^{M}_{i}/G_{i}$ is generated in degrees $> \alpha_{i+1}$. 
\end{defn}
For instance, in the example in \eqref{eqn:5D}, the linear strand
\[
S^{10}\gets S^{15}(-1)\gets S^{10}(-2)\gets 0.
\]
of $F^M$ is a numerical subcomplex of $F^M$ determined by $\alpha = (1,2,3,4)$.

For the proof of Theorem~\ref{thm:puresexist}, we use a numerical subcomplex $F^M$ to construct a submodule $M'\subseteq M$, where $\beta(M')=c_0\pi_{d^0}$.  Defining the appropriate numerical subcomplex and the submodule $M'$ will be relatively straightforward. However, since numerical complexes generally fail to be exact, it is not a priori clear that we should be able to determine the Betti diagram of the submodule $M'$.  This computation relies on a \defi{monotonicity principle} about the Betti numbers of pure diagrams.  

\begin{theorem}[Monotonicity principle]\label{thm:monotonicity}
Suppose that $d,e$ are degree sequences with $d_{i} = e_{i}$ and $d_{i+1}=e_{i+1}$. If $d<e$ then 
$$
\frac{\beta_{i,d_i}(\pi_{d})}{\beta_{i+1,d_{i+1}}(\pi_{d})} < \frac{\beta_{i,e_i}(\pi_{e})}{\beta_{i+1,e_{i+1}}(\pi_{e})}
$$
\end{theorem}
This theorem turns out to be surprisingly powerful, and we apply it to compute the Betti diagram of our submodule $M'\subseteq M$.  This Monotonicity Principle is related to some of the numerical inequalities for pure diagrams from \cite[Lemma~4.1]{mccullough} and \cite[\S3]{erman-beh}.

This paper is organized as follows.  In the next section we provide the necessary background on \BS theory. In \S\ref{sec:North fork}--\S\ref{sec:obtaining}, we develop our technique for producing cleanly embedded submodules.  We then discuss some limitations and extensions of our main result in  \S\ref{sec:beyond}.  The last two sections are devoted to the applications described above.

\section{Notation and Background on Boij-S\"oderberg Theory}\label{sec:background}
Throughout, all modules are assumed to be finitely generated, graded $S$-modules.  We use $(F^M,\phi^M)$ to refer to the minimal resolution of a module $M$, though we may omit the upper index $M$ in cases where confusion is unlikely.  

\begin{defn}\label{defn:Ff numerical complex}
Fix a module $M$, a minimal free resolution $(F^M, \phi^M)$ of $M$, and a sequence of integers $\mathbf{f}=(f_0,\dots,f_n)\in \ZZ^{n+1}$.  We define $(F(\mathbf{f})^M,\phi(\mathbf{f})^M)$ to be a sequence of free modules and maps
\[
\xymatrix{
\cdots \ar[r]& F(\mathbf{f})^M_i \ar[r]^-{\phi(\mathbf{f})^M}&F(\mathbf{f})^M_{i-1} \ar[r]&\cdots
}
\]
 as follows. 
Let $\iota_i: F(\mathbf{f})^M_i\to F_i^M$ be the inclusion of the graded free submodule consisting of all free summands of $F^M_i$ generated in degrees $< f_i$, and let  $\pi_i: F_i^M\to F(\mathbf{f})^M_i$ be a splitting of $\iota_{i}$ whose kernel consists of  
free summands generated in degrees $\geq f_{i}$. Finally, set
\[
\phi(\mathbf{f})^M_i =  \pi_{i-1}\circ \phi^M_i \circ \iota_i\colon  F(\mathbf{f})^M_i\to  F(\mathbf{f})^M_{i-1}.
\]
\end{defn}
\noindent Note that $F(\mathbf{f})^M$ is not necessarily a complex (see Example~\ref{ex:not complex}).

\begin{example}
Let
\[
\beta(F^{M})=
\begin{pmatrix}
{12}&{26}&{16}&-\\
-&-&-&{1}\\
-&5&-&{1}\\
-&-&12&17
\end{pmatrix}.
\]
Then  $F((1,3,5,6))^{M}$ is a numerical subcomplex  with Betti diagram
\[
\beta(F(1,3,5,6)^{M})=
\begin{pmatrix}
{12}&{26}&{16}&-\\
-&-&-&{1}\\
-&-&-&{1}\\
-&-&-&-
\end{pmatrix}.
\]
This is the largest numerical subcomplex containing only the linear first syzygies.
\end{example}
\begin{example}\label{ex:not complex}
For $S=k[x,y,z]$, let $M=S/(x,y,z^2)$.  Then
\[
\beta(M)=\begin{pmatrix}
1&2&1&-\\
-&1&2&1
\end{pmatrix}.
\]
We have 
\[ \beta(F(1,2,4,5)^{M})=\begin{pmatrix}
1&2&1&-\\
-&-&2&1
\end{pmatrix},
\]
but $F(1,2,4,5)^{M}$ is not a complex since
\[
\phi(1,2,4,5)^M_1=\begin{pmatrix} x&y \end{pmatrix} \quad \text{ and }\quad \phi(1,2,4,5)^M_2=\begin{pmatrix} 0&z^2&-y\\-z^2&0&x \end{pmatrix}
\]
do not compose to $0$.
\end{example}

We think of a Betti diagram $\beta(M)$ as an element of the infinite dimensional $\mathbb Q$-vector space $\mathbb V:=\oplus_{i=0}^n \oplus_{j\in \mathbb Z} \mathbb Q$.  The \defi{semigroup of Betti diagrams} $\Bmod$ is the subsemigroup of $\mathbb V$ generated by $\beta(M)$ for all modules $M$.  We define the \defi{cone of Betti diagrams} $\BQ$ as the positive cone spanned by $\Bmod$ in $\mathbb V$, and we define $\Bint$ as the semigroup of lattice points in $\BQ$.  See \cite{erman-semigroup} for comparisons between $\Bint$ and $\Bmod$.

Boij-S\"oderberg theory describes the cone $\BQ$.\footnote{There 
is  also a``dual'' side of the theory that describes the cone of cohomology diagrams of vector bundles and coherent sheaves on $\mathbb P^n$; see~\cite{eis-schrey1, eis-schrey2}.}
As conjectured in \cite{boij-sod1} and proven in~\cite{eis-schrey1, boij-sod2}, the extremal rays of $\BQ$ are spanned by
\defi{pure diagrams} $\pi_{d}$ (as defined above in \eqref{eqn:pid}) where $d = (d_{0},\dots,d_{n})\in (\ZZ\cup \{+\infty\})^{n+1}$ is a degree sequence, i.e. $d_i+1 \le d_{i+1}$.
We will also use the notation $\widetilde{\pi}_d$ for the smallest integral point on the ray spanned by $\pi_d$. So $\widetilde{\pi}_d=m \pi_d$
with $m=lcm(\prod_{k \not=i,k\le c} |d_i-d_k| , i=0,\ldots, t)$ where $t= \max\{i \mid d_i < \infty\}$ is the length of the degree sequence.

The cone $\BQ$ has the structure of a simplicial fan: if we partially order the sequences $d$ termwise, then there is a unique decomposition of any $\beta(M)\in \BQ$ as
\begin{equation}\label{eqn:BSdecomp}
\beta(M)=\sum_{i=0}^s c_{i}\pi_{d^i}
\end{equation}
with $c_{i}\in \mathbb Q_{\geq 0}$ and $d^{0}  <\cdots< d^{s}$.
We refer to this as the \defi{Boij-S\"oderberg decomposition} of $\beta(M)$.   For an expository account of Boij-S\"oderberg theory, see one of~\cite{eis-schrey-icm,floystad}.

If $\Delta=(d^0, \dots, d^s)$ is a chain of degree sequences $d^0<d^1<\dots<d^s$, then we use the notation $\BQ(\Delta), \Bint(\Delta)$ and $\Bmod(\Delta)$ for the restrictions of $\BQ, \Bint,$ and $\Bmod$ to the simplicial cone generated by the pure diagrams whose degree sequences lie in $\Delta$.   When $D\in \BQ(\Delta)$ with $\Delta=(d^0, \dots, d^s)$, the \defi{top strand} of $D$ consists of the entries parametrized by $d^0$, namely $\left(\beta_{0,d^0}(D), \beta_{1,d^0_1}(D), \dots, \beta_{n,d^0_n}(D)\right)$.  We refer to $c_{0}\pi_{d^0}$ as the first step of the Boij-S\"oderberg decomposition, and so on.  We will repeatedly use the fact that the algorithm for decomposing any such $D$ proceeds as a greedy algorithm on the top strand of $D\in \BQ$.  See~\cite[\S1]{eis-schrey1} for details.
 `
\begin{defn}\label{defn:pure filtration}
A \defi{full clean filtration} of a finitely generated graded $S$-module $M$ is
a sequence of cleanly embedded submodules
\[
M=M_0\supsetneq M_1\supsetneq M_2\supsetneq \dots \supsetneq M_t=0
\]
such that each $M_i/M_{i+1}$ has a pure resolution.
\end{defn}

It is immediate that we can put together full clean filtrations in extensions:
\begin{lemma}\label{lem:pures extensions}
Let $M'\subseteq M$ be a cleanly embedded submodule, and let $M''=M/M'$.  If $M'$ and $M''$ admit full clean filtrations, then so does $M$.\qed
\end{lemma}

Many numerical invariants of $M$ may be computed in terms of the Betti diagram of $M$, including the projective dimension of $M$, the depth of $M$, the Hilbert polynomial of $M$, and more.  We extend all such numerical notions to arbitrary diagrams $D\in \mathbb V$.  For instance, we say that the diagram
\[
D=\begin{pmatrix}
1&\frac{8}{3}&2&-\\-&-&-&\frac{1}{3}
\end{pmatrix}
\]
has projective dimension $3$.

When $M$ has finite length, we use the notation $M^\vee$ for the graded dual module $\Hom(M,k)$.

\section{The North fork of $F^M$}\label{sec:North fork}
We begin the construction of cleanly embedded submodules by studying the maximal numerical subcomplex of $F^M$ that contains only the first syzygies of minimal degree.  For instance, let $M$ be any module such that 
\begin{equation}\label{eqn:101510}
\beta(M)=\begin{pmatrix}
10 & 15 & 10 & - \\
- & 15 &15 &-\\
- & 10&15&10
\end{pmatrix}.
\end{equation}
$M$ is generated entirely in degree $0$, and $M$ has some linear first syzygies.  In this case, the maximal numerical subcomplex of $F^M$ containing these linear first syzygies is the linear strand of $F^M$, which corresponds to $F(\mathbf{f})^M$ where $\mathbf{f}=(1,2,3,5)$:
\[
\NF^M: S^{10}\gets S(-1)^{15}\gets S(-2)^{10}\gets 0.
\]

This type of numerical subcomplex plays an important role for us, and we refer to it as the \defi{North fork of $F^M$}.  This name is meant to suggest that $\NF^M$ consists of the part of the complex that ``flows through'' the minimal degree first syzygies.  The following definition states this more formally.

\begin{defn}\label{defn:North forkresolutions}
The \defi{North fork of $F^M$} is $(\NF^M, \phi(\mathbf{f})^M)$, where $\mathbf{f}$ is defined as follows: Let $f_0$ be one more than the maximal degree of a generator of $M$ and let $f_1$ be one more than the minimal degree of a first syzygy of $M$.  For $i>1$, set
\begin{equation}\label{eqn:fi}
f_i:=\min\{j | j>f_{i-1},\text{ and } \beta_{i,j}(M)\ne 0\}.
\end{equation}
\end{defn}
Note that $f_i>f_{i-1}$ with the possible exception that $f_1$ might be smaller than or equal to $f_0$.  Allowing $f_1\leq f_0$ slightly streamlines our argument in the case of a module generated in multiple degrees.  Namely,
since all generators of $F_0$ have degree $<f_0$, it follows that $\phi^M_1$ has the block form $\phi^M_1=\begin{pmatrix} \phi(\mathbf{f})^M_0&b_1^M\end{pmatrix}$.
\begin{lemma}
The North fork of $F^M$ is a complex.  
\end{lemma}
\begin{proof}
By splitting the inclusions $\NF^M_i\to F^M_i$, we may decompose each $\phi^M_i$ as $\phi^M_i=\begin{pmatrix} a^M_i & b^M_i\end{pmatrix}$, where the source $a^M_i$ is $\NF^M_i$.  Since $\NF^M_i$ consists of all summands generated in 
degree $< f_{i}$, the image of $a^M_i$ does not depend on the choice of basis for $F_i$.

From the definition of the $f_i$ it follows that $a^M_i$ factors through the inclusion $\NF_{i-1}^M\to F^M_{i-1}$.  As in Definition~\ref{defn:numerical subcomplex}, we use $\phi(\mathbf{f})^M_i$ to denote the induced map $\phi(\mathbf{f})^M_i: \NF^M_i\to \NF^M_{i-1}$.  We may thus rewrite $\phi^M_i$ as a block upper triangular matrix:
\begin{equation}\label{eqn:blocktriangle}
\phi^M_i=
\bordermatrix{
&\deg <f_i&\deg \geq f_i\cr
\deg < f_{i-1}&\phi(\mathbf{f})^M_i&*\cr
\deg \geq f_{i-1}&0&*
}
\end{equation}
for all $i>1$.  It follows immediately that $(\NF^M, \phi(\mathbf{f})^M)$ is actually a complex. 
\end{proof}

\begin{example}\label{ex:more101510}
Let $M$ be as in \eqref{eqn:101510}. The Betti diagram of $N:=\coker(\phi(\mathbf{f})^M_1)$ has the form
\[
\beta(N)=\begin{pmatrix}
10&15&10&-\\
-&-&*&-\\
-&-&*&*\\
-&-&\vdots & \vdots
\end{pmatrix},
\]
where $*$ indicates and unknown entry.  The $\beta_{3,3}$ and $\beta_{3,4}$ entries of $\beta(N)$ are zero because any low-degree third syzygy of $N$ would lift to a third syzygy of $M$ (see Lemma~A.2 for more details).
\end{example}

In Section~\ref{sec:obtaining} we shall show that, under the hypotheses of Theorem~\ref{thm:puresexist}, the cleanly embedded submodule of $M$ whose
existence is asserted by the theorem is the module $H^0_{\mathfrak m}(N)$, where $N:=\coker(\phi(\mathbf{f})^M_1)$. 
The following lemma relates the resolution of a module defined via the North fork or a similar construction, like the module $\coker(\phi(\mathbf{f})^M_1)$, to the resolution of the original module $M$.  (The unusual numbering of the following lemma matches the published version of this paper, where this lemma appeared in a separate corrigendum.)
\begin{lemmaA2}\label{lem:degs}
Let $M$ be a module and let $\mathbf{f}=(f_0,\dots,f_n)\in \ZZ^{n+1}$ be any sequence such that $f_0$ is greater than the maximal degree of a generator of $M$, and such that
\[
f_2:=\min\{j | j>f_1 \text{ and } \beta_{2,j}(M)\ne 0\}.
\]
Let $N$ be the cokernel of $\phi(\mathbf{f})^M$, as defined in Definition~\ref{defn:Ff numerical complex}, and let $K$ be the kernel of the natural surjection $N\to M$.
\begin{enumerate}
	\item\label{lem:degs:i}  $\Tor_i(K,k)$ is generated in degree $\geq f_2+i-1$.
	\item\label{lem:degs:ii}  If $e\leq f_2+i-2$ then we get an injection:
$
\Tor_i(N,k)_e\to \Tor_i(M,k)_e
$
	\item\label{lem:degs:iii}  If $e<f_2+i-2$ then we also get a surjection:
$
\Tor_i(N,k)_e\to \Tor_i(M,k)_e
$
\end{enumerate}
\end{lemmaA2}
\begin{proof}
By definition of $N$, we have that $\Tor_1(N,k)\subseteq \Tor_1(M,k)$, so the long exact sequence in Tor induces:
\[
\dots \to \Tor_2(M,k)\to \Tor_1(K,k)\overset{0}{\to} \Tor_1(N,k)\to \dots
\]
Let $m$ be the minimal degree of a generator of $\Tor_1(K,k)$.  Since $K$ is generated in degree $\geq f_1$ (see Definition~\ref{defn:Ff numerical complex}), it follows that $m>f_1$; the surjectivity of $\Tor_2(M,k)\to \Tor_1(K,k)$ then implies that $\beta_{2,m}(M)\ne 0$.  Hence $m\geq f_2$ and \eqref{lem:degs:i} follows immediately.

For  \eqref{lem:degs:ii} and  \eqref{lem:degs:iii}, we consider:
\[
\cdots \to \Tor_i(K,k)_e \to \Tor_i(N,k)_e \to \Tor_i(M,k)_e\to \Tor_{i-1}(K,k)_e\to \cdots
\]
If $e\leq f_2+i-2$, then \eqref{lem:degs:i} implies that $\Tor_i(K,k)_e=0$, which proves \eqref{lem:degs:ii}.  If $e< f_2+i-2$, then \eqref{lem:degs:i} implies $\Tor_{i-1}(K,k)_e=0$, which yields \eqref{lem:degs:iii}.
\end{proof}

\section{The Monotonicity Principle and its application}\label{sec:determining}
 \begin{proof}[Proof of Theorem~\ref{thm:monotonicity}]
If $d$ and $e$ have the same length as degree sequences, then, by inserting a maximal chain of degree sequences between $d$ and $e$, we see that it is enough to treat the case
where $d_{k} = e_{k}$ for all but one value of $k$, which cannot be equal to $i$ or to $i+1$, and where $e_k=d_k+1$.
In view of the Herzog-K\"uhl equations \eqref{eqn:pid}, the desired inequality is 
$$
\frac {|d_{k}-d_{i+1}|} {|d_{k}-d_{i}|} < \frac {|1+d_{k}-d_{i+1}|} {|1+d_{k}-d_{i}|}.
$$
If $k>i+1$ then $0<d_{k}-d_{i+1} < d_{k}-d_{i}$, so the result has the form
$$
\frac{a}{b} <\frac{a+1}{b+1} 
$$
where $0<a<b$, and this is immediate. In the case $k<i$, on the other hand,  we have $d_{i+1} - d_{k} > d_{i}-d_{k}> d_i-d_k-1>0$, so the result has the form
$$
\frac{a}{b} <\frac{a-1}{b-1} 
$$
with $a>b>1$, and again this is immediate.

If $d$ and $e$ have different lengths as degree sequences, then we can immediately reduce to the case $d=(d_0,\dots, d_t)\in \ZZ^{t+1}$ and $e=(d_0,\dots,d_{t-1},\infty)\in (\ZZ\cup \{\infty\})^{t+1}$.  In this case, we set $d^\ell:=(d_0,d_1,\dots,d_{t-1},d_t+\ell)$ for all $\ell\in \mathbb N$.  A direct computation via \eqref{eqn:pid} yields:
\[
\pi_e=\lim_{\ell \to \infty}\ell \cdot  \pi_{d^{\ell}}.
\]
Since all of the degree sequences $d^\ell$ have length $t$, we conclude that
\[
\frac{\beta_{i,d_i}(\pi_{d})}{\beta_{i,d_{i+1}}(\pi_{d})} <\frac{\beta_{i,d_i}(\pi_{d^1})}{\beta_{i,d_{i+1}}(\pi_{d^1})} 
<\dots<\frac{\beta_{i,d_i}(\pi_{d^\ell})}{\beta_{i,d_{i+1}}(\pi_{d^\ell})} <\frac{\beta_{i,d_i}(\pi_{d^{\ell+1}})}{\beta_{i,d_{i+1}}(\pi_{d^{\ell+1}})} 
<\dots< \frac{\beta_{i,e_i}(\pi_{e})}{\beta_{i,e_{i+1}}(\pi_{e})}.
\]
\end{proof}

The next example shows how the Monotonicity Principle can be used to determine Betti diagrams.
\begin{example}\label{ex:BSbetaNorthFork}
Consider $M$ and $N$ as in Example~\ref{ex:more101510}.  Recall that the Betti diagram of $N$ has the form
\begin{equation}\label{eqn:Dpartial}
\beta(N)=\begin{pmatrix}
10 & 15 & 10 & - \\
- & - &* &-\\
- & -&*&*\\
-&-&\vdots &\vdots 
\end{pmatrix}.
\end{equation}
Can one determine the remaining entries of the above Betti diagram from the given information? 

Since we know that $\NF^M$ is a numerical subcomplex of the minimal free resolution of $N$, we at least know something about the top strand of $\beta(N)$.  One can thus attempt to compute the first Boij-S\"oderberg summand of $\beta(N)$. With the Monotonicity Principle this approach leads to a complete determination of $\beta(N)$ as follows.

If $\pi_d$ is a diagram that could appear in the Boij-S\"oderberg decomposition of $\beta(N)$ and which contributes to either the $\beta_{1,1}$ or $\beta_{2,2}$ entry, then $d$ must have the form $(0,1,d_2,d_3)$ with $2\leq d_2$ and $5\leq d_3$.  The minimal such $d$ is $d=(0,1,2,5)$, and by applying the formula \eqref{eqn:pid}, we see
\[
\frac{\beta_{1,1}(\pi_{(0,1,2,5)})}{\beta_{2,2}(\pi_{(0,1,2,5)})}= \frac{15}{10}.
\]
Note that this equals the ratio $\frac{\beta_{1,1}(N)}{\beta_{2,2}(N)}$.  

Now, the Monotonicity Principle implies that if $e=(0,1,2,d_3)$ with $d_3>5$ then \linebreak $\beta_{1,1}(\pi_{e})/\beta_{2,2}(\pi_{e})> 15/10.$  If we allow $e$ to have the form $e=(0,1,d_2,d_3)$ with $d_2>2$, then $\pi_e$ does not have any $\beta_{2,2}$ entry, and so the ratio would be $\infty$.  We conclude that every pure diagram $\pi_d$ which could conceivably contribute to $\beta_{1,1}(N)$ satisfies $\frac{\beta_{1,1}(\pi_{d})}{\beta_{2,2}(\pi_{d})}\geq \frac{15}{10},$ with equality if and only if $d=(0,1,2,5)$.

Since the decomposition algorithm implies that we cannot eliminate $\beta_{1,1}$ before we eliminate $\beta_{2,2}$, it follows that we must eliminate both entries simultaneously.  Thus, the first step of the \BS decomposition of $\beta(N)$ is given by $1\cdot \widetilde{\pi}_{d^0}=1\cdot \widetilde{\pi}_{(0,1,2,5)}$.  

Continuing to apply the decomposition algorithm, we next consider the diagram $\beta(N)-1\cdot \widetilde\pi_{d^0}$, which has the form
\[
\beta(N)-1\cdot \widetilde\pi_{d^0}=\begin{pmatrix}
4 & - & - & - \\
- & - &* &-\\
- & -&*&*\\
-&-&\vdots &\vdots 
\end{pmatrix}.
\]
Since the second column consists of all zeroes, this diagram must be $4 \pi_{(0)}$.  Hence,
\[
\beta(N)=\widetilde\pi_{(0,1,2,5)}+4\widetilde\pi_{(0)}=\begin{pmatrix}
10 & 15 & 10 & - \\
- & - &- &-\\
- & -&-&1\\ 
\end{pmatrix}.
\]
\qed
\end{example}

We will generally apply the monotonicity principle via the following corollary.  However, as illustrated by Example~\ref{ex:branching} and by the computations in \S\ref{sec:gensBmod}, the principle can be useful in more general situations.

\begin{cor}\label{cor:0free}
Let $M$ be a module satisfying the hypotheses of Theorem~\ref{thm:puresexist}\eqref{thm:puresexist:1}, and let $\NF^M$ be the North fork of $F^M$.  Set $N:=\coker(\phi(\mathbf{f})^M_1)$.  We may write
\[
\beta(N)=c_0\pi_{d^0}+D_{\text{free}}
\]
where $D_{\text{free}}$ is the Betti diagram of a free module.
\end{cor}
\begin{proof}[Proof of Corollary \ref{cor:0free}]
We combine definition~\ref{defn:North forkresolutions} and the fact that $d^0_2\leq d^1_1$ (which is part of the definition $d^0\ll d^1$) to conclude that $f_2=d^1_2$.  Lemma~A.2 then implies that
\[
\beta_{i,d^0_i(M)}=\beta_{i,d^0(N)} \quad \text{ for } i=1,2.
 \]

Next, again since $d^0\ll d^1$, we have $d^0_1<d^0_{2}\leq d^1_1$ and $d^0_{2}\leq d^1_1< d^1_{2}$.  It follows that $\pi_{d^0}$ is the only pure diagram from the \BS decomposition of $\beta(M)$ that contributes to the Betti numbers $\beta_{1,d^0_1}(M)$ and $\beta_{2,d^0_{2}}(M)$.  This implies the second equality of
\begin{equation}\label{eqn:ratio equal}
\frac{\beta_{1,d^0_1}(N)}{\beta_{2,d^0_2}(N)}=\frac{\beta_{1,d^0_1}(M)}{\beta_{2,d^0_2}(M)}=\frac{\beta_{1,d^0_1}(\pi_{d^0})}{\beta_{2,d^0_2}(\pi_{d^0})},
\end{equation}
and the first equality follows from the first paragraph of this proof.

Now, let $e_i$ be the minimal degree of a generator of the $i$th syzygy module of $N$, or $\infty$ if this syzygy module is $0$.  We claim that the degree sequence $e=(e_0,\dots, e_n)$ is at least $d^0$.  Since all first syzygies of $N$ lie in degree $d^0_1$, this would imply that $e_1=d^0_1$ or $e_1=\infty$.  For contradiction, assume that $e_i<d^0_i$ for some $i\geq 2$.  The definition of $d^0\ll d^1$ in \eqref{defn:ll} implies that
\[
e_i<d^0_i\leq d^1_2+i-1=f_2+i-1.
\] 
Since $e_i\leq f_2+i-2$, Lemma~A.2\eqref{lem:degs:ii} implies that $\Tor_i(M,k)_{e_i}$ is nonzero.  This yields a contradiction, since $\Tor_i(M,k)_{e_i}$ is nonzero by definition of $e_i$, and hence $e_i$ cannot be small than $d^0_i$, yielding the claim.

Now, if $e_2=d^0_2$ but $e\ne d^0$, then by Theorem~\ref{thm:monotonicity} combined with \eqref{eqn:ratio equal}, we have that
\[
\frac{\beta_{1,d^0_1}(N)}{\beta_{2,d^0_2}(N)}<\frac{\beta_{1,d^0_1}(\pi_{e})}{\beta_{2,d^0_2}(\pi_{e})}.
\]
If $e_2\ne d^0_2$, then the denominator on the right would be $0$.

By convexity, the only sums $\sum_e a_e\pi_e$ which satisfy \eqref{eqn:ratio equal} are rational linear combinations of $\pi_{d^0}$ and of projective dimension $0$ pure diagrams $\pi_{(e_0,\infty,\dots,\infty)}$.  Finally, since $\beta_{1,d^0_1}(N)=\beta_{1,d^0_1}(M)$, we conclude that the coefficient of $\pi_{d^0}$ in the \BS decomposition of $\beta(N)$  equals $c_0$.  This completes the proof.
\end{proof}

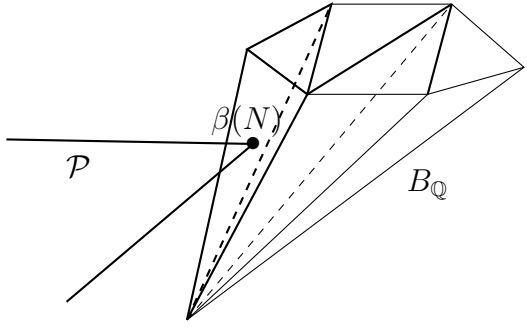
\begin{figure}
\centering
\begin{tikzpicture}[xscale=1.6,yscale=1.2]
\draw[-,thick] (2,2.5)--(3.2,3.5);
\draw[-,thick] (3,2.5)--(3.2,3.5);
\draw[-,thick](1,0)--(1.5,3);
\draw[-,thick](1,0)--(2,2.5);
\draw[-](1,0)--(3,2.5);
\draw[-](1,0)--(3.8,2.8);
\draw[dashed,-](1,0)--(3.2,3.5);
\draw[dashed,-,thick](1,0)--(2.2,3.5);
\draw[-,thick](1.5,3)--(2.2,3.5)--(2,2.5)--cycle;
\draw[-](1.5,3)--(2,2.5)--(3,2.5)--(3.8,2.8)--(3.2,3.5)--(2.2,3.5)--cycle;
\draw (3,1.5) node {$B_{\mathbb Q}$};
\draw (1.5,2.2) node {$\beta(N)$};
\draw (1.55,1.95) node{$\bullet$};
\draw (.1,1.7) node {$\mathcal P$};
\draw[-,thick](0,.2)--(1.55,1.95)--(-0.5,2);
\end{tikzpicture}
\caption{This figure is a sketch of the situation of Remark~\ref{rmk:polyhedron}.  Our partial information about $\beta(N)$  corresponds to a polyhedron $\mathcal P$.  Since $\mathcal P\cap \BQ$ consists of a single point, this actually determines all of  $\beta(N)$.}
\label{fig:polyhedron}
\end{figure}

\begin{remark}\label{rmk:polyhedron}
The idea behind Example~\ref{ex:BSbetaNorthFork} and Corollary~\ref{cor:0free} may be illustrated by convex geometry.  Our goal is to understand where in the cone $\BQ$ the diagram $\beta(N)$ lies.  As illustrated in \eqref{eqn:Dpartial}, we only have partial knowledge about $\beta(N)$.  We can think of this partial information as cutting out a polyhedron $\mathcal P$ in the vector space $\mathbb V$, and the diagram $\beta(N)$ must lie in the intersection of $\mathcal P$ and $\BQ$.  The computation in Example~\ref{ex:BSbetaNorthFork}  then shows that $\mathcal P\cap \BQ$ consists of a single point (see Figure~\ref{fig:polyhedron}), which is how we determine the remaining entries of $\beta(N)$.
\end{remark}

\section{Proof of Theorem~\ref{thm:puresexist} and Corollary~\ref{cor:puresexist-iterated}}\label{sec:obtaining}

We begin by showing that under suitable hypotheses the conclusion of Corollary~\ref{cor:0free} implies an actual splitting of $N$:
\begin{lemma}\label{lem:0free}
If $N$ is a module such that
\[
\beta(N)=D_{\geq 2}+D_{\text{free}}
\]
where $D_{\geq 2}$ is a diagram of codimension $\geq 2$ and $D_{\text{free}}$ is a diagram of projective dimension $0$, and such that
$$
\min \{j \mid \beta_{0,j}(D_{\text{free}})\ne 0\} \geq 
\max  \{j \mid \beta_{0,j}(D_{\geq 2})\ne 0\},
$$
then $N$ splits as a direct sum
$
N\cong N_{\geq 2}\oplus N_{\text{free}}
$
 with $\beta(N_{\geq 2})=D_{\geq 2}$ and $\beta(N_{\text{free}})=D_{\text{free}}$.
\end{lemma}

Informally, the displayed inequality above says that the minimum degree of a ``generator'' of $D_{\text{free}}$ is at least as large as the maximum degree of a ``generator'' of $D_{\geq 2}$.

\begin{proof}
Let $a:=\max \{ j | \beta_{0,j}(N)\ne 0\}$, the maximal degree of a minimal generator of $N$. Let $K$ be the quotient field of $S$.  By considering the Hilbert polynomial of $N$, we see that $N\otimes_S K$  has rank $\geq 1$, and thus some minimal generator of degree $a$ in
$N$ generates a free submodule. This gives us an exact sequence
\[
0\to S(-a)\to N \to Q\to 0.
\]
The map $S(-a)\to N$ lifts to a map $S(-a)\to F^{N}_{0}$ whose image is a free summand,
so $\beta(Q)$ satisfies the same hypothesis as $\beta(N)$. By induction on the number of generators, we see that $Q$ is a direct sum of a free module $G$ and a module $H$ of codimension $\geq 2$. Since $\Ext^{1}(G, S) = \Ext^{1}(H,S) = 0$, the sequence splits.
\end{proof}
\begin{example}
The inequality appearing in Lemma~\ref{lem:0free} is necessary.  For instance, let $S=k[x,y]$ and let $N:=S(-1)\oplus S/(x^2,xy)$.  Then
\[
\beta(N)=\begin{pmatrix}
1&-&-\\
1&2&1
\end{pmatrix}=
\begin{pmatrix}-&-&-\\
1&2&1
\end{pmatrix}
+
\begin{pmatrix}1&-&-\\
-&-&-
\end{pmatrix}.
\]
Thus $\beta(N)$ has the form $D_{\geq 2}+D_{\text{free}}$.  But $N\ncong S\oplus S(-1)/(x,y)$.
\end{example}
\begin{example}
The conclusion of Lemma~\ref{lem:0free} may fail without the hypothesis ``codimension $\geq 2$''.  For instance, if $S=k[x,y]$ and $\mathfrak m=(x,y)$, then
\[
\beta(\mathfrak m)=\begin{pmatrix}
2&1
\end{pmatrix}=
\begin{pmatrix}1&1
\end{pmatrix}
+
\begin{pmatrix}1&-
\end{pmatrix},
\]
but $\mathfrak m$ does not split.
\end{example}

We are now ready to complete the proof of our main result:

\begin{proof}[Proof of Theorem~\ref{thm:puresexist}.]
We first prove part \eqref{thm:puresexist:1}.  We let $\NF^M$ be the North fork of $F^M$, and we define $N:=\coker(\phi(\mathbf{f})^M_1)$.  Since $M$ satisfies the hypotheses of Theorem~\ref{thm:puresexist}, we may apply Corollary~\ref{cor:0free} and Lemma~\ref{lem:0free}, and conclude that $N=M'\oplus G$ where $\beta(M')=c_0\pi_{d^0}$ and $G$ is free.

We may then rewrite $\phi(\mathbf{f})^M_1$ as a block matrix $\phi(\mathbf{f})^M_1=\begin{pmatrix} \widetilde{a}_1 \\ 0 \end{pmatrix}$, where $\widetilde{a}_1$ is a minimal presentation matrix of $M'$.  This enables us to rewrite $\phi_1$ in upper triangular form:
\[
\phi_1=\begin{pmatrix}\widetilde{a}_1 & \widetilde{b}_1\\ 0 & \widetilde{c}_1  \end{pmatrix}
\]
for some matrices $\widetilde{b}_1$ and $\widetilde{c}_1$.  Since $M$ is presented by a block triangular matrix, we obtain a right exact sequence:
\[
M'\to M\to \coker(\widetilde{c}_1)\to 0.
\]

To finish the proof, we will apply Lemma~A.2 to show that this sequence is exact on the left and that $M'$ is a cleanly embedded submodule.  
Note that, by the definition of $N$ we have $f_2=d^1_2$, which satisfies the conditions of Lemma~A.2 by definition of $d^0\ll d^1$ (see~\eqref{defn:ll} on page \pageref{defn:ll}).  In fact, the definition of $d^0\ll d^1$ further implies that $d^0_i\leq d^1_2+i-2$ for all $i\geq 1$.  For $i>2$ this is part of the definition, and the fact that $d^0<d^1$ and $d^0_2\leq d^1_1$ imply that $d^0_i\leq d^1_2+i-2$ for $i\leq 2$ as well.  Thus, for any $i\geq 1$, Lemma~A.2(2) implies that $\Tor_i(N,k)_{d^0_i}$ (which equals $\Tor_i(M',k)_{d^0_i}$) injects into $\Tor_i(M,k)_{d^0_i}$.  Since $M'$ and $M$ are both finite length and since $M'$ has a pure resolution, the inclusion in the case $i=n$ implies that $M'\to M$ is injective as claimed; the other inclusions on $\Tor$ groups imply that $M'\subseteq M$ is cleanly embedded, completing the proof of \eqref{thm:puresexist:1}.

For~\eqref{thm:puresexist:2}, since $d^0\ll d^1$ we obtain a cleanly embedded submodule $M'\subseteq M$ with $\beta(M')=c_0\pi_{d^0}$.  Set $M'':=M/M'$.  The sequence
\[
0\to M'\to M\to M''\to 0
\]
corresponds to an element $\alpha\in \Ext^1(M'',M')$, which then corresponds to a cocycle $\alpha_0\in \Hom(F_1^{M''}, M')$.  
Since
$F^{M''}_1$ is generated in degree at least $d^1_1$, it follows that the image of the map $\alpha_0$ is generated in degree at least $d^1_1$.
However, since $\beta(M')=c_0\pi_{d^0}$, we see that $M'$ has regularity $d^0_{n}-n$, and thus is zero in degrees $>d^0_{n}-n$.  By our assumption 
\[
d^1_1> d^0_n-n,
\]
so the image of $\alpha_0$ is $0$.  We conclude that $\alpha$ corresponds to the zero element of $\Ext^1(M'',M')$, and thus that $M\cong M'\oplus M''$, as desired.
\end{proof}

\begin{proof}[Proof of Corollary~\ref{cor:puresexist-iterated}] 
With notation as in Theorem~\ref{thm:puresexist}, we choose $M^{1} = M'$.  The proof of Theorem~\ref{thm:puresexist} shows that, for degree reasons, the induced map $F^{M^{1}}_j \to F^{M}_j$ is a split injection for all $j$. It follows that 
$$
\beta(M/M^{1}) = \sum_{i=1}^s c_{i}\pi_{d^i},
$$
so we may iterate the construction.
\end{proof}

\begin{example}
There exist cases covered by Corollary~\ref{cor:puresexist-iterated} where a full clean filtration exists, but where that filtration is not a splitting: Let $S=k[x,y,z]$ and let $\Phi$ be a generic $9\times 9$ skew-symmetric matrix of linear forms. Let $I\subseteq S$ be the ideal generated by the $8\times 8$ principal Pfaffians of $\Phi$, and let $R=S/I$.  Then $R$ has a pure resolution of type $(0,4,5,9)$.  We claim that if $M$ is a generic extension
\[
0\to R\to M\to R(-2)\to 0,
\]
then $M$ admits a full clean filtration which is not a splitting.

Note first that, for any such extension, $R\to M$ is cleanly embedded for degree reasons.  Namely, if we construct a resolution of $M$  by combining the resolutions of $R$ and $R(-2)$, then there is no possibility of cancellation.  It thus suffices to show that $\Ext^1(R,R)_2\ne 0$.  
Such an extension corresponds to a nonzero map $\alpha_0: F^{R(-2)}_1=S(-6)^9\to R$ such that $\alpha_0\circ \Phi=0$.  Since $R$ has regularity $6$ and $\im(\alpha_0\circ \Phi)\subseteq R_7=0$, we see that $\alpha_0\circ \Phi$ is automatically $0$.  One may easily check that there exists such an $\alpha_0$ that is not a coboundary.
\end{example}

\begin{example}\label{ex:old example}
For $n>2$, fix any $e\geq 2$, and let $M$ be any module such that
$
\beta(M)$ decomposes as a sum of the pure diagrams $\pi_{(0,e,e+1,e+2,\dots,e+n-2,e+n-1)}$ and $\pi_{(0,1,2,\dots,n-1,e+n-1)}.$
Then $M$ has a Betti diagram of the form:
\[
\beta(M)=\begin{pmatrix}
*&*&*&\dots &*&-\\
-&-&-&\dots &-&-\\
\vdots&\vdots&& &\vdots&\vdots\\
-&-&-&\dots &-&-\\
-&*&*&\dots &*&*
\end{pmatrix}.
\]
Theorem~\ref{thm:puresexist}\eqref{thm:puresexist:2} implies that $M$ splits as $M=M'\oplus M''$ where $M'$ has a pure resolution of type $(0,e,e+1,e+2,\dots,e+n-2,e+n-1)$ and $M''$ has a pure resolution of type $(0,1,2,\dots,n-1,e+n-1)$.  Note that every $S$-module with a pure resolution of type $(0,e,e+1,e+2,\dots,e+n-2,e+n-1)$ is a direct sum of copies of $R:=S/\mathfrak m^e$.  It follows that $M'$  is isomorphic to a direct sum of copies of $R$.  By a similar argument, $M''$ is isomorphic to a number of copies of $\omega_R(n)$.  Hence, any such $M$ decomposes as $M=R^a\oplus \omega_R(n)^b$ for some $a,b$.  
\end{example}

\section{Beyond Theorem~\ref{thm:puresexist}}\label{sec:beyond}

Since the Boij-S\"oderberg decomposition of a module may involve pure diagrams with non-integral entries, it is clear that there exist many graded modules which do not admit full clean filtrations. 
\begin{example}\label{ex:no filtration}
Let $n=2$, $R=k[x,y]/(x,y)^2,$ and $M=k[x,y]/(x,y^2)$.  Then:
\[
\beta(M)=\begin{pmatrix} 1&1&-\\-&1&1\end{pmatrix} = \frac{1}{3}\beta(R)+\frac{1}{3}\beta(\omega_R(4)).
\]
Clearly $M$ cannot admit a full clean filtration.  Though we might hope that $M^{\oplus 3}$ admits such a filtration, this is not the case either~\cite[Ex.\ 4.5]{sam-weyman}.

However, there does exist a flat deformation $M'$ of $M^{\oplus 3}$ such that $M'$ admits a full clean filtration:
\[
0 \to R \to M' \to \omega_R(4) \to 0.
\]
Namely, we may set $M'=\left( S/(x,y^2) \right) \oplus \left( S/(x^2,y) \right)\oplus \left( S/(x+y,(x^2-2y+y^2) \right)$.  This suggests a more subtle possible affirmative answer to our Question~\ref{question:beyond}. \qed
\end{example}

Each result of \S\ref{sec:North fork}--\ref{sec:obtaining} can be extended to situations that are not covered by the hypotheses of Theorem~\ref{thm:puresexist}.  

\begin{example}\label{ex:branching}
Let $E:=\widetilde\pi_{(0,2,3,4,5,8)}+2\widetilde\pi_{(0,2,3,5,6,8)}+\widetilde\pi_{(0,3,4,5,6,8)}+\widetilde\pi_{(0,3,4,6,7,8)}$, and let $M$ be a module such that $\beta(M)=E$. We have
\[
E=\begin{pmatrix}
11 & - & - & - & -&-\\
- & 60 &128 &90 &32&-\\
- & 144 &300 &128 &60&-\\
- & - & -& 280&240&69
\end{pmatrix}.
\]
Note that the degree sequences do not satisfy the conditions of Corollary~\ref{cor:puresexist-iterated}.   Nevertheless, we will see that that $M$ admits a full clean filtration.

We first construct a cleanly embedded (but not pure) submodule of $M$.  We let $\NF^M$ be the North fork of $F^M$ and we let $N:=\coker(\phi(\mathbf{f})^M_1)$.  The proof of Corollary~\ref{cor:0free} applies nearly verbatim to yield
\[
\beta(N)=\widetilde\pi_{(0,2,3,4,5,8)}+2\widetilde\pi_{(0,2,3,5,6,8)}+6\widetilde\pi_{(0)}.
\]
By Lemma~\ref{lem:0free}, we obtain a splitting $N=M'\oplus G$ where $G$ is a free module.  The arguments used in the proof of Theorem~\ref{thm:puresexist} then imply that $M'$ is a cleanly embedded submodule of $M$.  We thus have a short exact sequence
\[
0\to M'\to M\to M''\to 0
\]
where $\beta(M')=\widetilde\pi_{(0,2,3,4,5,8)}+2\widetilde\pi_{(0,2,3,5,6,8)}$ and $\beta(M'')=\widetilde\pi_{(0,3,4,5,6,8)}+\widetilde\pi_{(0,3,4,6,7,8)}$.  

Repeating the same argument for $(M')^\vee$ and for $(M'')^\vee$, and then applying Lemma~\ref{lem:pures extensions}, we conclude that $M$ admits a full clean filtration.
\end{example}

One of the key features of our proof of Theorem~\ref{thm:puresexist} is that the diagrams $d^0, \dots, d^s$ that arise in the Boij-S\"oderberg decomposition are separated from each other in the poset of degree sequences.  In particular, $d^0$ and $d^{1}$ always differ in at least two consecutive positions.  This is essential to our proof of Corollary~\ref{cor:0free}, and it suggests some interesting examples to explore.

Consider, for example, the diagrams $D=\widetilde\pi_{(0,1,3,5)}+\widetilde\pi_{(0,2,4,5)}$ and $D'=\widetilde\pi_{(0,1,2,3,5,6)}+\widetilde\pi_{(0,1,3,4,5,6)}$, so that
\[
D=\begin{pmatrix}
11&15&-&-\\
-&10&10&-\\
-&-&15&11
\end{pmatrix}
\qquad
\text{ and }
\qquad
D'=
\begin{pmatrix}
3&12&15&10&-&-\\
-&-&10&15&12&3
\end{pmatrix}.
\]
\begin{question}
Let $\beta(M)$ be a scalar multiple of either $D$ or $D'$.  Does $M$ admit a cleanly embedded submodule with a pure resolution?
\end{question}

\begin{remark}
Although many aspects of our technique apply to modules of dimension greater than $0$, there is one obstacle to extending our results to such modules.  Let $M$ be a module of nonzero Krull dimension that otherwise satisfies the hypotheses of Theorem~\ref{thm:puresexist}, and define $N$ via the North fork as in the proof of Theorem~\ref{thm:puresexist}.  It is possible that the projective dimension of $N$ could be larger than the projective dimension of $M$, and this possibility undermines our application of the monotonicity principle.  It would thus be interesting to produce a positive answer to Question~\ref{question:beyond1} for some case where the dimension of $M$ is nonzero.
\end{remark}
%

\section{Application:  Pathologies of $\Bmod$}\label{sec:pathologies}
Example~\ref{ex:first filtration} illustrates the existence of a ray of $\BQ$ where only $\frac{1}{5}$ of the lattice points correspond to Betti diagrams of modules.  We now prove Proposition~\ref{prop:sparserays}, which implies that there are rays where the true Betti diagrams are arbitrarily sparse among the lattice points.  The proof will show that such pathologies already arise in codimension $3$.
\begin{proof}[Proof of Proposition~\ref{prop:sparserays}]
Let $S=k[x_1, x_2, x_3]$ and let $p\geq 5$ prime.   Set $d^0=(0,1,2,p), d^1=(0,\lfloor p/2\rfloor, \lceil p/2 \rceil, p)$ and $d^2=(0,p-2,p-1,p)$.  Consider the diagram
\[
D=\frac{1}{p}\widetilde{\pi}_{d^0}+\frac{\alpha}{p}\widetilde{\pi}_{d^1}+\frac{1}{p}\widetilde{\pi}_{d^2}
\]
where $\alpha$ is any positive integer such $\alpha+1+\binom{p-1}{2}\equiv 0 \mod p$.  We claim that $D$ has integral entries but that $cD\in \Bmod$ if and only if $c$ is divisible by $p$.

We first check the integrality of $D$.  Observe that each Betti number of $\widetilde{\pi}_{d^0}$ is divisible by $p$ except for the $0$th Betti number;  each Betti number of $\widetilde{\pi}_{d^2}$ is divisible by $p$ except for the $3$rd Betti number; and the Betti numbers of $\widetilde{\pi}_{d^1}$ are $(1,p,p,1)$.  Hence, we only need to check that $\beta_{0,0}(D)$ and $\beta_{3,p}(D)$ are integral.  We compute
\[
\beta_{0,0}(D)=\frac{1}{p}\beta_{0,0}(\widetilde{\pi}_{d^0})+\frac{\alpha}{p}\beta_{0,0}(\widetilde{\pi}_{d^1})+\frac{1}{p}\beta_{0,0}(\widetilde{\pi}_{d^2})
=\frac{1}{p}+\frac{\alpha}{p}+\frac{\binom{p-1}{2}}{p}.
\]
Our assumption on $\alpha$ then implies that $\beta_{0,0}(D)$ is integral.  A symmetric computation works for $\beta_{3,p}(D)$.

Let $cD\in B_{\text{mod}}$, and we will show that $c$ is divisible by $p$.  Let $M$ be any module such that $\beta(M)=cD$.  Let $N$ be the cokernel of the submatrix of the presentation matrix of $M$ containing all of the degree $1$ and degree $\lfloor p/2\rfloor$ columns (so we throw away the degree $p-2$ columns).   Lemma~A.2\eqref{lem:degs:ii} and \eqref{lem:degs:iii} then imply that $\beta_{i,j}(N)=\beta_{i,j}(M)$ for $i=2,3$ and $j<p+i-3$.

In particular, the top strand of $\beta(N)$ is at least $(0,1,2,p)$, and so we may use the monotonicity principle to conclude that the first step of the Boij--S\"oderberg decomposition of $\beta(N)$ is given by $\frac{c}{p}\widetilde{\pi}_{d^0}$.  The top strand of the resulting diagram $\beta(N)-\frac{c}{p}\widetilde{\pi}_{d^0}$ is then at least $(0,\lfloor p/2\rfloor,\lceil p/2\rceil,p)$, and an additional application of the monotonicity principle yields the full Boij--S\"oderberg decomposition of $\beta(N)$ to be:
\[
\beta(N)=\frac{c}{p}\widetilde{\pi}_{d^0}+\frac{c\alpha}{p}\widetilde{\pi}_{d^1}+\frac{c}{p}\pi_{(0,\infty,\dots,\infty)}.
\]
By Lemma~5.1, $N$ splits off a free summand $S^{\frac{c}{p}}$, and hence $p$ divides $c$.

It now suffices to show that $pD\in \Bmod$.  This follows from the fact that $\widetilde{\pi}_{d^i}\in \Bmod$ for $i=0,1$ or $2$.  In particular, $\widetilde{\pi}_{d^2}=\beta(R)$ where $R:=S/(x_1,x_2,x_3)^{p-2}$, and $\widetilde{\pi}_{d^0}=\beta(R^\vee(p-3))$.  To see that $\widetilde{\pi}_{d^1}\in \Bmod$, let $A$ be a $p\times p$ skew-symmetric matrix of generic linear forms.  By~\cite{buchs-eis-codim3}, the principal Pfaffians of $A$ define an ideal $I\subseteq S$ such that $\beta(S/I)=\widetilde{\pi}_{d^1}$.

This completes the proof when $p\geq 5$.  For the cases $p=2$ (respectively $3$), we may choose the diagram $D=\frac{1}{2}\widetilde{\pi}_{(0,1,2,4)}+\frac{1}{2}\widetilde{\pi}_{(0,2,3,4)}$ (respectively $D=\frac{1}{3}\widetilde{\pi}_{(0,1,2,5)}+\frac{2}{3}\widetilde{\pi}_{(0,3,4,5)}$) and apply similar arguments as above.
\end{proof}

\section{Application: Quiver representations}\label{sec:gensBmod}
In this section, we determine all Betti diagrams corresponding to quiver representations of the form $
 \bullet \triplearrow{}{} \bullet$.  As discussed in the introduction, this is equivalent to computing the possible Betti diagrams of finite length modules of the form:
 \[
\beta(M)=\begin{pmatrix} \beta_{0,0} & \beta_{1,1} & \beta_{2,2} & \beta_{3,3} \\
\beta_{0,1} & \beta_{1,2} & \beta_{2,3} & \beta_{3,4}
\end{pmatrix}.
\]
Throughout this section, we thus set $d^0=(0,1,2,3), d^1=(0,1,2,4), d^2=(0,1,3,4), d^3=(0,2,3,4)$ and $d^4=(1,2,3,4)$ and we let $\widetilde{\Delta}=(d^0,d^1,d^2,d^3,d^4)$.  

Our goal is to compute the minimal generators of $\Bmod(\widetilde{\Delta})$. In addition to the connection with quiver representations, this computation provides the first detailed and nontrivial example of the generators of $\Bmod(\widetilde{\Delta})$. Further, this computation illustrates that the Monotonicity Principle and some of the other techniques introduced in \S\ref{sec:North fork}-\ref{sec:obtaining} can be extended to more situations, but at the cost of wrestling with integrality conditions and precise numerics.

As noted in the introduction, if $\beta_{3,3}(M)$ (or $\beta_{0,1}(M)$) is nonzero, then a copy of the residue field $k$ (or $k(-1)$) splits from $M$.  It is therefore equivalent to restrict to the case where $\beta_{3,3}=\beta_{0,1}=0$ and to compute the generators for $\Bmod(\Delta)$ where $\Delta=(d^1,d^2,d^3)$.  The result of this computation is summarized in the following proposition.
\begin{prop}\label{prop:gensBmod}
The semigroup $\Bmod(\Delta)$ has ten minimal generators.  These consist of the following ten Betti diagrams:
\[
\begin{pmatrix}
3&8&6&-\\
-&-&-&1
\end{pmatrix},
\begin{pmatrix}
1&-&-&-\\
-&6&8&3
\end{pmatrix},
\begin{pmatrix}
1&2&1&-\\
-&1&2&1
\end{pmatrix},
\begin{pmatrix}
1&1&-&-\\
-&3&5&2
\end{pmatrix},
\begin{pmatrix}
2&5&3&-\\
-&-&1&1
\end{pmatrix},
\]
\
\[
\begin{pmatrix}
2&4&1&-\\
-&1&4&2
\end{pmatrix},
\begin{pmatrix}
3&7&3&-\\
-&-&3&2
\end{pmatrix},
\begin{pmatrix}
2&3&-&-\\
-&3&7&3
\end{pmatrix},
\begin{pmatrix}
2&4&-&-\\
-&-&4&2
\end{pmatrix}
\begin{pmatrix}
3&6&-&-\\
-&-&6&3
\end{pmatrix}
.
\]
\end{prop}
See Figure~\ref{fig:slice}.
Before proving this proposition, we introduce some simplifying notation.  Every element of $\Bint(\Delta)$ can be represented as:
\[
D=4r\pi_{(0,1,2,4)}+2s\pi_{(0,1,3,4)}+4t\pi_{(0,2,3,4)}
\]
with $(r,s,t)\in \mathbb Z_{\geq 0}^3$
(c.f. \cite[pp.\ 347--9]{erman-semigroup}.)
The necessary and sufficient conditions for a triplet $(r,s,t)\in \mathbb Z_{\geq 0}^3$ to yield an integral point are:
\begin{itemize}
    \item  $r+s\equiv 0 \mod 3$
    \item  $s+t\equiv 0 \mod 3$
    \item  $r+s+ t \equiv 0 \mod 2$.
\end{itemize}
For the rest of this section, we use triplets $(r,s,t)$ to refer to diagrams in $\Bint(\Delta)$, and we only consider triplets $(r,s,t)$ that satisfy the above congruency conditions.  In this notation, Proposition~\ref{prop:gensBmod} amounts to the claim that the following ten $(r,s,t)$ triplets are the generators of $\Bmod$:
\[
\begin{matrix}
(6,0,0), (0,0,6), (1,2,1), (3,3,0), (0,3,3), \\
 (1,8,1), (3,9,0), (0,9,3), (0,12,0), (0,18,0).
\end{matrix}
\]


\begin{figure}
\usetikzlibrary[patterns]
\begin{tikzpicture}
\filldraw (0,0) circle (2pt);
\draw (-.5,-.5) node {$4c\pi_{(0,1,2,4)}$};
\filldraw (2,3.4) circle (2pt);
\draw (2,3.9) node {$4c\pi_{(0,2,3,4)}$};
\filldraw (4,0) circle (2pt);
\draw (4,-.5) node {$2c\pi_{(0,1,3,4)}$};
\draw[thin] (0,0)--(2,3.4)--(4,0)--(0,0);
\draw[thin, dashed] (.4,0)--(2.16,3.14);
\draw[thick,->] (.3,1.9)--(1.1,1.6);
\draw[thick,->] (3.7,1.9)--(2.9,1.6);
\draw(-1.3,1.9) node {The region $s<5$};
\draw(5.3,1.9) node {The region $s\geq 5$};
\fill[fill=gray,semitransparent](.4,0)--(2.16,3.14)--(4,0);]
\fill[pattern=dots, semitransparent] (0,0)--(.4,0)--(2.16,3.14)--(2,3.4);
\end{tikzpicture}
\caption{Proposition~\ref{prop:gensBmod} can be illustrated by considering a slice of the cone $\BQ(\Delta)$ where $r+s+t=c$ for some $c\gg 0$.  In the region where $s<5$, roughly half of the lattice points in the cone belong to $\Bmod$.  In the region where $s\geq 5$, every lattice point in the cone belongs to $\Bmod$.}\label{fig:slice}
\end{figure}
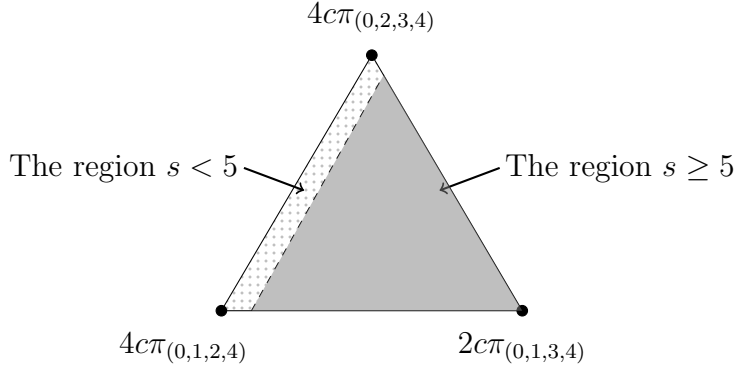

\begin{proof}[Proof of Proposition~\ref{prop:gensBmod}]
We first note that each of the ten diagrams listed in Proposition~\ref{prop:gensBmod} is the Betti diagram of an actual module.  When $\beta_{0,0}=1$ or $\beta_{3,4}=1$, such examples are straightforward to construct.  Next, we have
\[
\beta \left(\coker \begin{pmatrix}
x&y&0&0&z^2\\
0&x&y&z&x^2
\end{pmatrix}
\right)
=
\begin{pmatrix}
2&4&1&-\\
-&1&4&2
\end{pmatrix}.
\]
Let $L$ be any $2\times 3$ matrix of linear forms whose columns satisfy no linear syzygies, and let $N:=\coker(L)$.  Then
\[
\beta(N/\mathfrak m^2N)=\begin{pmatrix}2&3&-&-\\-&3&7&3  \end{pmatrix}
\]
The Betti diagram of $\left( N/\mathfrak m^2N\right)^\vee$ then yields the dual diagram.  Finally, examples corresponding to $(0,12,0)$ and $(0,18,0)$ are given in \cite[Proof of Thm.\ 1.6(1)]{erman-semigroup}.

We must now show that every diagram in $\Bmod(\Delta)$ may be written as a sum of our ten generators.  We proceed by analyzing cases based on the different possible values of $s$ in our $(r,s,t)$ representation of diagrams.
\subsection*{The case $s=0$}
Based on Example~\ref{ex:old example} in the case $n=3$ and $e=2$, we conclude that $(r,0,t)$ corresponds to an element of $\Bmod(\Delta)$ if and only if both $r$ and $t$ are divisible by $6$.

\subsection*{The case $s=1$}
There are two families of triplets $(r,1,t)$ satisfying the congruency conditions.  The first family is parametrized by $(2+6\gamma,1,5+6\alpha)$ for some $\gamma,\alpha \in \mathbb Z_{\geq 0}$, and the second family is parametrized by $(5+6\gamma,1,2+6\alpha)$.  To prove that none of these diagrams belongs to $\Bmod(\Delta)$, it suffices (by symmetry under $M\mapsto M^\vee$) to rule out the first family.

We thus assume, for contradiction, that there exists $M$ such that $\beta(M)$ corresponds to the triplet $(2+6\gamma, 1, 5+6\alpha)$ for some $\alpha, \gamma\in \mathbb Z_{\geq 0}$.  We let $F(\mathbf{f})^M$ be the North fork of $F^M$.  We then set $N:=\coker( \phi(\mathbf{f})^M_1)$, and we have
\[
\beta(N) =\begin{pmatrix}2+\alpha+3\gamma & 3+8\gamma & 2+6\gamma& - \\
- & - & \beta_{2,3}(N)& \beta_{3,4}(N)\\
- & - & \beta_{2,4}(N) & \beta_{3,5}(N)\\
-&-&\vdots & \vdots
 \end{pmatrix}.
\]
To produce the Boij-S\"oderberg decomposition, we begin by subtracting $c_1\pi_{d^1}$ for some $c_1\geq 0$.   Note that
\[
c_1\pi_{d^1}=c_1\begin{pmatrix}\frac{1}{8}&\frac{1}{3}&\frac{1}{4}&-\\-&-&-&\frac{1}{24}  \end{pmatrix}.
\]

Assume first that $c_1<24\gamma$, so that $\frac{c_1}{4}<6\gamma$.  In this case, the greedy decomposition algorithm must eliminate the $\beta_{3,4}$ first (or $c_1=0$ and $\beta_{3,4}=0$ to begin with).  The resulting diagram $\beta(N)-c_1\pi_{d^1}$ has $\beta_{3,4}=0$ and the ratio $\beta_{1,1}/\beta_{2,2}$ is strictly larger than $3/2=\beta_{1,1}(\pi_{(0,1,2,5)})/\beta_{1,1}(\pi_{(0,1,2,5)})$.  The monotonicity implies that this is impossible.
Hence, we must have $c_1\geq 24\gamma$.

If now $c_1=24\gamma$, then we may again apply the monotonicity principle to $\beta(N)-24\gamma\pi_{d^1}$ to conclude that the next step of the Boij-S\"oderberg decomposition must be precisely $\frac{1}{5}\pi_{(0,1,2,5)}$.  This would leave nothing left in column $1$, and thus $\beta(N)-\gamma\pi_{d^1}-\frac{1}{5}\pi_{(0,1,2,5)}$ would be a diagram of projective dimension $0$.  But this would contradict the integrality of $\beta(N)$, since it would imply that $\beta_{3,5}(N)=\frac{1}{5}$.

The final possibility is that $c_1>24\gamma$.  Since $c_1<48\gamma$, this implies that we must eliminate the $\beta_{2,2}$ entry first and hence that $c_1$ must equal $8+24\gamma$.  After subtracting $(8+24\gamma)\pi_{d^1}$, we are left with:
\[
\beta(N)-\left(8 +24\gamma\right)\pi_{(d^1)}
=
\begin{pmatrix}1+\alpha& \frac{1}{3} &- & -\\
 - & -& \beta_{2,3}(N) & \beta_{3,4}(N)-(\frac{1}{3}+\gamma)\\
 -&-&\vdots&\vdots
 \end{pmatrix}.
\]
Since $\beta_{3,4}(N)-(\frac{1}{3}+\gamma)$ is nonzero (it is not an integer), the next step of the Boij-S\"oderberg decomposition must eliminate this entry.  This means that the next step of the decomposition must be $\frac{1}{6}\pi_{d^2}$.  However, this would leave a $0$ in column $1$ and a nonzero entry in column $2$, which is impossible.


\subsection*{The case $s=2$}
There are two families of triplets $(r,2,t)$ satisfying the congruency conditions.  The first family has the form $(1+6\gamma, 2, 1+6\alpha)$ and the second family has the form $(4+6\gamma, 2, 4+6\alpha)$, where $\gamma, \alpha \in \mathbb Z_{\geq 0}$.  Every element of the first family is a sum of our proposed generators, so we must show that no element of the second family belongs to $\Bmod(\Delta)$.  We obtain a contradiction by essentially the same analysis as in the case $s=1$.

\subsection*{The case $s=4$}
There are two families of triplets $(r,4,t)$ satisfying the congruency conditions, namely $(2+6\gamma, 4, 2+6\alpha)$ and $(5+6\gamma, 4, 5+6\alpha)$.  Since every element of the first family is a sum of our proposed generators, we must show that no element of the second family belongs to $\Bmod(\Delta)$.  A similar, though more involved, analysis as in the case $s=1$ then illustrates that there are no such diagrams.

\subsection*{The cases $s=3,5,6$}
We claim that if $D\in \Bint(\Delta)$ corresponds to an $(r,s,t)$-triplet where $s=3,5,$ or $6$, then $D\in \Bmod(\Delta)$, with the exception of $(0,6,0)$.  There are six families to consider in total: $(3+6\gamma,3,6\alpha), (6\gamma,3,3+6\alpha), (4+6\gamma,5,1+6\alpha), (1+6\gamma,5,4+6\alpha), (3+6\gamma,6,3+6\alpha),$ and $(6\gamma,6,6\alpha)$.  Any element from any of these families may be written as a sum of our proposed generators, except for $(0,6,0)$.  The diagram corresponding to $(0,6,0)$ does not belong to $\Bmod$ by~\cite[Proof of Thm.\ 1.6(1)]{erman-semigroup}.

\subsection*{The cases $s>6$}
One may directly check that all elements of $\Bint(\Delta)$ with $s>6$ can be written as an integral sum of the proposed generators.
\end{proof}

\section*{Acknowledgements}
We thank Christine Berkesch, Jesse Burke, Courtney Gibbons, Steven Sam and Jerzy Weyman for useful conversations.  We thank the anonymous referee for a careful reading that improved the paper.  We also thank Amin Nematbakhsh and to Gunnar Fl{\o}ystad for bringing to our attention a mistake in the proof of Corollary~\ref{cor:0free} in a previous version.  The computer
algebra system \texttt{Macaulay2} \cite{M2} provided valuable
assistance in studying examples. 

\bibliographystyle{amsalpha}
\bibliography{filtering.bib}

\end{document}